\documentclass[11pt, a4paper]{amsart}

\usepackage[margin=0.75in]{geometry}
\usepackage{amsfonts}
\usepackage{amssymb}
\usepackage{amsthm}
\usepackage{amsaddr}
\usepackage{amsmath}
\usepackage{mathtools}
\usepackage{thmtools, thm-restate}

\usepackage[utf8]{inputenc}

\usepackage{chngcntr}
\usepackage{booktabs}

\usepackage{subcaption}

\newtheorem{theorem}{Theorem}[section]
\newtheorem{proposition}[theorem]{Proposition}
\newtheorem{lemma}[theorem]{Lemma}

\newtheorem{assumption}[theorem]{Assumption}
\newtheorem{remark}[theorem]{Remark}

\newtheorem{definition}[theorem]{Definition}

\counterwithin{figure}{section}
\counterwithin{table}{section}
\counterwithin{equation}{section}
\counterwithin{theorem}{section}
\counterwithin{assumption}{section}

\newcommand{\mc}{\mathcal}

\renewcommand{\dim}{\operatorname{dim}}
\newcommand{\dpair}[2]{\left\langle#1,#2\right\rangle}
\newcommand{\rank}[1]{\operatorname{rank}\left(#1\right)}
\newcommand{\clos}[2]{\overline{#2}^{#1}}
\newcommand{\linspan}[1]{\operatorname{span}\left\{#1\right\}}
\renewcommand{\Im}[1]{\operatorname{Im}\left(#1\right)}
\newcommand{\bs}{\boldsymbol}
\renewcommand{\r}{\bs r}
\newcommand{\T}{\mc T_{\r}}
\newcommand{\id}{\operatorname{id}}
\newcommand{\R}{\mathbb R}

\newcommand{\N}{\mathbb N}

\newcommand{\G}{\mathbb G}
\newcommand{\Umin}{U^{\min}}
\newcommand{\mix}{\text{mix}}

\newcommand{\cref}{\ref}
\newcommand{\Cref}{\ref}

\begin{document}

\title{Singular Value Decomposition
       in Sobolev Spaces: Part I}

\author{Mazen Ali and Anthony Nouy}

\address{M. Ali: Ulm University, Institute for Numerical Mathematics}
\email{mazen.ali@uni-ulm.de}
\address{A. Nouy: Centrale Nantes, LMJL (UMR CNRS 6629)}
\email{anthony.nouy@ec-nantes.fr}

\begin{abstract}
A well known result from functional analysis states that any compact operator
between Hilbert spaces admits a
singular value decomposition (SVD). This decomposition
is a powerful tool that is the workhorse
of many methods both in mathematics and applied fields. A prominent
application in recent years is the approximation of
high-dimensional functions in a low-rank format.
This is based on the fact that, under certain conditions, a tensor
can be identified with a compact operator and SVD applies to the latter.
One key assumption for this application is that the tensor product norm
is not weaker than the injective norm. This assumption is not fulfilled in
Sobolev spaces, which are widely used in the theory and numerics of
partial differential equations. Our goal is the analysis of the
SVD in Sobolev spaces.

This work consists of two parts. In this manuscript (part I), we address low-rank
approximations and minimal subspaces in $H^1$. We analyze the $H^1$-error of
the SVD performed
in the ambient $L^2$-space. In part II, we will address variants of
the SVD in norms stronger than the $L^2$-norm.
We will provide a few numerical examples
that support our theoretical findings.
\end{abstract}



\maketitle

\section{Introduction}
\subsection{Low-Rank Approximation}
We begin by providing a few motivating examples where the need for
low-rank approximation arises. We do not aim to provide a comprehensive
overview and refer instead
to, e.g., \cite{Khoromskaia2018, Khoromskij2018, HK05, HB, Nouy2017}
for more details.

Consider a prototypical situation in numerical approximation where
a function $f\in C([0,\,1]^d)$ is approximated by a discrete function
on a finite grid
\begin{align}\label{eq:fd}
    f_{k_1,\,\ldots,\,k_d} := f(x_{k_1,\,\ldots,\,x_d}),\quad
    x_{k_1,\,\ldots,\,x_d} := \left(k_1h,\,\ldots,\,
    k_dh\right)\in [0,\,1]^d,\quad 0\leq k_j\leq n,
\end{align}
where $h=1/n$ and $n+1$ is the number of grid points in each dimension.
Simply storing such an approximation requires storing at least
$(n+1)^d$ values -- and this number grows exponentially in $d$.
This quickly becomes unfeasible for problems with a large dimension $d$.

Low-rank approximation is a widely used tool to address high-dimensional
problems. Suppose we can approximate $f$ in the form
\begin{align*}
    f\approx f_r:=\sum_{k=1}^r f_k^1\otimes\ldots\otimes f_k^d.
\end{align*}
Storing or evaluating the right-hand-side $f_r$
requires only the knowledge of the
$rd$ one-dimensional entries $f_k^j$ and thus reduces the cost to $nrd$.
Thus, this is a significant reduction in cost, \emph{provided $r$ is
small} in some sense (hence, the term \emph{low-rank}).

Prominent applications where such high-dimensional problems appear arise
in quantum mechanics and quantum chemistry. Wave functions describing physical
states of non-relativistic quantum systems are functions in the space
$L^2(\R^{3N})$, where $N$ is the number of elementary particles. Thus,
for multi-particle systems the dimension $d=3N$ is large even for comparatively
simple models. A wave function $\psi\in L^2(\R^{3N})$ is then given as a solution
to a PDE, such as the Schrödinger equation, and $\psi$ is approximated
via, e.g., a finite difference method with point values as in \eqref{eq:fd},
or through the coefficients of some basis expansion
$\psi\approx\sum_{k_1,\,\ldots,\,k_d}c_{k_1,\,\ldots,\,k_d}
\varphi_{k_1,\,\ldots,\,k_d}$,
where $\varphi_{k_1,\,\ldots,\,k_d}$ is some multi-dimensional basis
of functions such as finite elements or wavelets.
See \cite{Khoromskaia2018} for more details.

Yet another application is the approximation of operators or inverse
operators for either solving or pre-conditioning equations. Suppose a
matrix or an operator are given in the form
\begin{align}\label{eq:A}
    A = \sum_{k=1}^d \id_1\otimes\ldots\otimes\id_{k-1}\otimes A_k
        \otimes\id_{k+1}\otimes\ldots\otimes\id_d.
\end{align}
Assume that a function $\varphi$ applied to $A$ has the integral
representation
\begin{align*}
    \varphi(A)=\int_\Omega\exp(AF(t))G(t)dt,
\end{align*}
for some $\Omega\subset\R$ and functions $F$ and $G$. This is valid for, e.g.,
the inverse function $\varphi(x)=x^{-\sigma}$ and a self-adjoint positive $A$.
Then, since $A$ has the form as in \eqref{eq:A}, a quadrature rule for
approximating the integral
\begin{align*}
    \varphi(x)=\int_\Omega\exp(xF(t))G(t)dt,\quad x\in\R,
\end{align*}
provides a separable approximation to $\varphi(A)$. For more details
we refer to \cite{Khoromskij2018, HK05}.

\subsection{The Singular Value Decomposition}
Let $T:H_1\rightarrow H_2$ be a continuous
compact linear operator between Hilbert spaces.
Then, for any $x\in H_1$
\begin{align}\label{eq:svd}
    Tx = \sum_{k=1}^\infty\sigma_k\dpair{x}{\psi_k}_{H_1}\phi_k,
\end{align}
for a non-negative non-increasing sequence $\{\sigma_k\}_{k\in\N}$ and orthonormal
systems\\
$\{\psi_k\}_{k\in\N}\subset H_1$ and
$\{\phi_k\}_{k\in\N}\subset H_2$.
The representation \eqref{eq:svd} is known as the \emph{singular value
decomposition} of $T$, or SVD for short. It is both a powerful analysis tool and
an approximation tool. Perhaps the most important feature of this decomposition
can be summarized as
\begin{align*}
    \left\|T-
    \sum_{k=1}^r \sigma_k\dpair{\cdot}{\psi_k}_{H_1}\phi_k\right\|=
    \sigma_{r+1}
    =\inf_{
    \rank{A}\leq r}\|T-A\|,
\end{align*}
where $\|\cdot\|$ refers to the standard operator norm and 
where the infimum is taken over all operators $A$ from $H_1$ to $H_2$ with rank bounded by $r$. I.e.,
\eqref{eq:svd} gives both the optimal approximation with rank $\leq r$
(for any $r$), obtained
by truncating the SVD, and the singular values $\{\sigma_k\}_{k\in\N}$
provide the best approximation errors.

SVD has many applications in both mathematics and applied
sciences. To name a few: computation of pseudoinverse, determination of rank
(and null space, range), least-squares minimization,
principal component analysis, proper orthogonal decomposition,
data compression, quantum entanglement.
For recent applications in model reduction see
\cite{RB, MOR}. The subject of this work is the application
of SVD to low-rank approximation of functions, see \cite{HB}.

A function $u$ in the tensor product $H_1\otimes H_2$ of two Hilbert spaces
$H_1$ and $H_2$ possesses a decomposition
\begin{align}\label{eq:introsvd}
    u=\sum_{k=1}^\infty\sigma_k\psi_k\otimes\phi_k,
\end{align}
if the norm on the tensor product space is not weaker than the injective norm. This
guarantees that $u$ can be identified with a compact operator and thus
\eqref{eq:svd} applies. These conditions are certainly satisfied for functions
between finite dimensional spaces. There are also important examples of
infinite dimensional spaces, where this is satisfied as well. The
most prominent example is the space of square integrable functions
$L^2(\Omega_1\times\Omega_2)$.

Low-rank approximations are of essential importance when dealing with
tensor product spaces $\bigotimes_{j=1}^d H_j$, $d\gg 2$.
There is no known generalization of SVD to $d>2$.
However, if we consider the vector space isomorphism (given appropriate norms),
$
    \bigotimes_{j=1}^d H_j\cong\left(\bigotimes_{j\in\alpha} H_j\right)
    \otimes\left(\bigotimes_{j\in \alpha^c} H_j\right),$
    $\alpha\subset\{1,\ldots, d\},
$
we can apply SVD in the latter tensor space
since this is again a two dimensional tensor product. This is known as
the \emph{higher-order singular value decomposition}
(HOSVD), see \cite{Lars, HBSEMA}. Thus, the theory for $d=2$ can be recycled for
higher dimensions. This applies to high-dimensional kernel operators
in $L^2(\bigtimes_{j=1}^d\Omega_j)$.

There are two other works\footnote{That we are aware of.} that considered the
related questions of regularity and error estimation of the $L^2$-SVD.
In \cite{Andre} the author showed that the $L^2$-SVD inherits the regularity
of the original function.
In \cite{Linf} the author investigated
$L^\infty$-error control of the $L^2$-SVD for functions with sufficient
smoothness by using the Gagliardo-Nirenberg inequality.

An important example where \eqref{eq:introsvd} does not apply are
multi-dimensional Sobolev spaces. The Sobolev norm on the tensor product space
is not weaker than the injective norm and thus Sobolev functions can not be
identified with compact operators. Another way of framing this from an
approximation standpoint: we can not
apply SVD to functions while controlling the Sobolev norm. However, not all
hope is lost, since Sobolev spaces are ``in between'' spaces where SVD
applies. E.g., the space of square integrable functions or the space of
functions with mixed smoothness. Moreover, Sobolev spaces such as $H^1(\Omega)$
can be identified with an intersection of tensor product spaces, where SVD
applies in each of the spaces in the intersection.

The purpose of this work is to analyze if and how SVD can be applied to
approximate functions in a Sobolev space. We work with the prototype
$H^1(\Omega)$, which frequently arises as the solution space of
partial differential equations. The results can be naturally extended to
the spaces $H^k(\Omega)$, $k>1$.
The paper is organized as follows. In Section \cref{sec:prelim} we briefly review
some of the basics of tensor spaces. In Section \cref{sec:approx} we discuss
low-rank approximation and minimal subspaces in Sobolev spaces.
In Section \cref{sec:l2svd} we analyze the $H^1$-error of the SVD performed
in the
ambient $L^2$ space ($L^2$-SVD).

\section{Preliminaries}\label{sec:prelim}
We briefly review some of the theory on tensor spaces and minimal subspaces.
Most of the following material can be found in \cite{HB}, some of it in
\cite{MTP}. We use the notation
$
    A\lesssim B$ $\Leftrightarrow$ $A\leq CB,
$
for some constant $C>0$ independent of $A$ or $B$. Similarly for
$\gtrsim$; and $\sim$ if both $\lesssim$ and $\gtrsim$ hold. We use
$\cong$ to denote vector space isomorphisms, with equivalent norms where
relevant.

\subsection{Algebraic Tensor Spaces}
Let $V=X\otimes_a Y$ be an algebraic tensor product space, where
$X$ and $Y$ are vector spaces. Briefly, it is the space of all
sums of the form
$
    v = \sum_{k=1}^r x\otimes y,$ $x\in X,\;y\in Y, \;r\in\N,
$
where the tensor product $\otimes$ is bilinear on $X\times Y$.
See \cite[Chapter 3.2]{HB} for a precise definition of the tensor product.

This construction can be extended for more than two vector spaces to obtain
the tensor space
$
    V={}_a\bigotimes_{j=1}^d X_j,
$
with elements
$
    v=\sum_{k=1}^r\bigotimes_{j=1}^dx_j,$ $x_j\in X_j,
    \;r\in\N.
$
We will sometimes require the isomorphic representations
\begin{align*}
    {}_a\bigotimes_{j=1}^d X_j\cong X_i\otimes_a\left(
    {}_a\bigotimes_{j\neq i}X_j
    \right)
    \cong\left({}_a\bigotimes_{j\in\alpha}X_j\right)\otimes_a
    \left({}_a\bigotimes_{i\in\alpha^c}X_i\right),
\end{align*}
where
$
    \alpha\subset\{1,\ldots,d\},\;\alpha^c=\{1,\ldots,d\}\setminus\alpha.
$

\subsection{Tensor Norms and Banach Tensor Spaces}
If we are given a norm $\|\cdot\|$ on the vector space
$V={}_a\bigotimes_{j=1}^d X_j$, we can consider the completion w.r.t.\ that
norm.

\begin{definition}[Topological Tensor Product]
    The space
    \begin{align*}
        {}_{\|\cdot\|}\bigotimes_{j=1}^d X_j:=
        \clos{\|\cdot\|}{{}_a\bigotimes_{j=1}^d X_j},
    \end{align*}
    is called a topological tensor product.
\end{definition}

Let each of the $X_j$ be a normed vector space. Since
$\|\cdot\|$ induces a topology on $V$ and with the product topology on
$\bigtimes_{j=1}^dX_j$, we can ask if
$
    \otimes:\bigtimes_{j=1}^dX_i\rightarrow V
$
is continuous. In fact, many useful properties in the analysis of tensor
product spaces require even stronger conditions. For ease of presentation,
we list the definitions for $d=2$.

\begin{definition}[Crossnorms]
    A norm on $V= X\otimes_a Y$ is called a crossnorm if
$
        \|x\otimes y\|=\|x\|_X\|y\|_Y.
$
    It is called a reasonable crossnorm if it is a crossnorm and
$
        \|x^*\otimes y^*\|^*=\|x^*\|_{X^*}\|y^*\|_{Y^*},$ $
        x^*\in X^*,\; y^*\in Y^*
$
    where $\|\cdot\|^*$ denotes the standard dual norm on the topological dual $Z^*$ of a space $Z$.

    It is called a uniform crossnorm if it is a reasonable crossnorm and
$
        \|A\otimes B\|=\|A\|\|B\|,$ $A\in\mc L(X, X),\;
        B\in \mc L(Y, Y),
$
    with the standard operator norms and where $\mc L(X, Y)$ denotes the space of
    continuous linear operators from $X$ to $Y$.
\end{definition}

There are two important examples of reasonable crossnorms which are
the strongest and the weakest crossnorms
(see \cite[Chapter 1.1.2]{MTP} for a justification of the terminology).
\begin{definition}[Projective and Injective Norms]\label{def:proinnrms}
    The projective norm on $V=X\otimes_a Y$ is defined as
    \begin{align*}
        \|v\|_{\wedge}:=
        \inf\left\{
        \sum_{i=1}^m\|x_i\|_X\|y_i\|_Y:
        v=\sum_{i=1}^mx_i\otimes y_i
        \right\},
    \end{align*}
    where the infimum is taken over all possible representations of $v$.
    The injective norm on $V=X\otimes_a Y$ is defined as
    \begin{align*}
        \|v\|_{\vee}:=\sup_{\varphi\in X^*\setminus\{0\},\;
        \psi\in Y^*\setminus\{0\}}\frac{|(\varphi\otimes\psi)v|}
        {\|\varphi\|_{X^*}\|\psi\|_{Y^*}}.
    \end{align*}
\end{definition}
By \cite[Proposition 4.68]{HB},
we have that for any reasonable crossnorm $\|\cdot\|$,
$
    \|\cdot\|_{\vee}\lesssim \|\cdot\|\lesssim\|\cdot\|_{\wedge}.
$
In this work we will frequently require the following definition.
\begin{definition}[Hilbert Tensor Space with Canonical Norm]
    Let $H=H_1\otimes_a H_2$ be an algebraic tensor product of two Hilbert
    spaces $H_1$ and $H_2$. The canonical inner product (and associated
    canonical norm) on $H$ is
    defined such that
$
        \dpair{x_1\otimes y_1}{x_2\otimes y_2}=
        \dpair{x_1}{x_2}_{H_1}\cdot\dpair{y_1}{y_2}_{H_2}.
$
    By linearity this definition extends to any $v\in H$. The canonical
    norm is a uniform crossnorm.
\end{definition}

\subsection{Sobolev Spaces}\label{sec:sobolev}
For the remainder of this work we will require the spaces
\begin{align*}
    L^2(\Omega)=L^2(\bigtimes_{j=1}^d\Omega_j),\quad
    H^1(\Omega)=H^1(\bigtimes_{j=1}^d\Omega_j).
\end{align*}
We use the shorthand notation $\|\cdot\|_0$ to denote the $L^2$ norm and
$\|\cdot\|_1$ to denote the $H^1$ norm. This notation
will be used both for the tensor product space and the one dimensional
components, where the difference should be clear from context.
We use $H^1_{\mix}(\Omega)$ to denote spaces of functions with mixed
smoothness with the corresponding norm $\|\cdot\|_{\mix}$.

We have
\begin{align*}
    L^2(\Omega)\cong{}_{\|\cdot\|_0}\bigotimes_{j=1}^dL^2(\Omega_j),\quad
    H^1_\mix(\Omega)\cong{}_{\|\cdot\|_\mix}\bigotimes_{j=1}^dH^1(\Omega_j),
\end{align*}
where $\|\cdot\|_0$ (resp.\ $\|\cdot\|_\mix$) are uniform crossnorms defined
from the norms $\|\cdot\|_0$ (resp.\ $\|\cdot\|_1$) on the individual
spaces $L^2(\Omega_j)$ (resp.\ $H^1(\Omega_j)$).

We frequently require spaces of functions differentiable in only one direction
\begin{align*}
    H^{e_k}:={H^1(\Omega_k)
    \otimes_{\|\cdot\|_{e_k}}\left({}_a\bigotimes_{j\neq k}L^2(\Omega_j)
    \right)},
\end{align*}
where for $e_k=(\delta_{1k},\ldots,\delta_{dk})$
being the $k$-th canonical vector, the norm is defined via
\begin{align*}
    \|v\|_{e_k}^2:=\|v\|_0^2+\left\|
    \frac{\partial}{\partial x_k}v
    \right\|_0^2.
\end{align*}
As in Definition
\cref{def:proinnrms}, we can define the projective and injective norms on
${H^1(\Omega_k)
{}_a\otimes\left({}_a\bigotimes_{j\neq k}L^2(\Omega_j)
\right)}$. We denote these norms by $\|\cdot\|_{\wedge(e_k)}$ and
$\|\cdot\|_{\vee(e_k)}$, respectively.
The space $H^1(\Omega)$ can be identified with the intersection space
\begin{align}\label{eq:sobolev}
    H^1(\Omega)\cong
    \bigcap_{k=1}^d H^{e_k},
\end{align}
where the latter is equipped with the intersection norm
$
    \|\cdot\|:=\max_{1\leq k\leq d}\|\cdot\|_{e_k},
$
or any equivalent norm. The utility in this representation lies in the
fact that $\|\cdot\|_{e_k}$ is the canonical norm on the Hilbert tensor space
$H^{e_k}$ and thus SVD applies (see Section \cref{sec:svd}).
For each $1\leq k\leq d$, we get a different decomposition.

\subsection{Minimal Subspaces and Tensor Formats}
For a tensor in the algebraic tensor space $X\otimes_a Y$, with $X$ and
$Y$ Hilbert spaces, the SVD gives the
representation
$
    v=\sum_{k=1}^r\sigma_k \psi_k\otimes\phi_k.
$
Letting
$
    U_1:=\linspan{\psi_k:1\leq k\leq r},$ $
    U_2:=\linspan{\phi_k:1\leq k\leq r},
$
we have the obvious statement $u\in U_1\otimes_a U_2$. More importantly,
these spaces are minimal in the sense that if $u\in V_1\otimes_a V_2$,
then $U_1\subset V_1$ and $U_2\subset V_2$. Spaces $U_1$ and $U_2$ are called the
\emph{minimal subspaces} of $u$ and they can be defined in a more general
setting.

\begin{definition}[Minimal Subspaces]\label{def:minsub}
    Let $\|\cdot\|\gtrsim \|\cdot\|_{\vee}$ be a norm on
    $V={}_a\bigotimes_{j=1}^d X_j$. For any
    $v\in\clos{\|\cdot\|}{V}$ the $j$-th minimal subspace is defined as
    \begin{align*}
        \Umin_j(v):=\clos{\|\cdot\|_{X_j}}
        {\linspan{\varphi(v):\varphi=\bigotimes_{k=1}^d\varphi_k,\;
        \varphi_j=\id_j,\;\varphi_k\in\left(X_k)^*,\;k\neq j\right)}},
    \end{align*}
    where $\id_j$ denotes the identity operator on $X_j$. 
    This definition can be naturally extended to $\Umin_{\alpha}(v)$ for any
    $\alpha\subset\{1,\ldots,d\}$.
\end{definition}

The question whether
\begin{align}\label{eq:inmin}
    v\in\clos{\|\cdot\|}{{}_a\bigotimes_{j=1}^d \Umin_j(v)}
\end{align}
is not trivial for topological tensors $v\in\clos{\|\cdot\|}{V}$. A positive
answer requires further structure of the component spaces and the tensor norm.

\begin{definition}[Grassmanian]
    Let $X$ be a Banach space. A closed subspace $U\subset X$ is called direct
    or complemented if there exists a closed subspace $W$ such that
    $X=U\oplus W$ is a direct sum. The set $\G(X)$ of all complemented subspaces in $X$
     is called the Grassmanian.
\end{definition}

Any closed subspace $U$ of a Hilbert space $X$ belongs to $\G(X)$.
An important example where \eqref{eq:inmin} is satisfied is when all
$X_j$ are Hilbert spaces and $\|\cdot\|$ is the canonical norm.
The Sobolev space $H^1(\Omega)$ does not have this property. In particular,
$\|\cdot\|\gtrsim\|\cdot\|_{\vee}$ does not hold. However,
$H^1(\Omega)$ is isomorphic to an intersection of tensor spaces, where
each individual space in the intersection satisfies \eqref{eq:inmin}.
This property is frequently exploited in our work.

Ultimately we are interested in low-rank approximations. For $d=2$, there is
only one choice of a low-rank format. However, for $d>2$ there are many
possible low-rank tensor formats.
The two most basic tensor formats are the following.
\begin{definition}[Canonical Format]\label{def:formatsCP}
    Let $r\in\N$. The $r$-term (or canonical) format in
    $V={}_a\bigotimes_{j=1}^d X_j$ is
    defined as
    \begin{align*}
        \mc R_r(V):=\left\{
        v=\sum_{k=1}^r\bigotimes_{j=1}^dx_j^k:x_j^k\in X_j
        \right\}.
    \end{align*}
\end{definition}

\begin{definition}[Tucker Format]\label{def:formatsTucker}
    For $\bs r=(r_1,\ldots, r_d)\in\N^d$, the Tucker format in
    $V={}_a\bigotimes_{j=1}^d X_j$ is defined as
    \begin{align*}
        \mc T_{\bs r}(V):=\left\{
        v\in V:\dim{\Umin_j(v)}\leq r_j
        \right\}.
    \end{align*}
\end{definition}

\subsection{Tensors as Operators and Singular Value Decomposition}\label{sec:svd}
Let
$\mc F(Y, X)$ denote the space of finite rank operators
from $Y$ to $X$, $\mc K(Y, X)$ denote the space of compact operators
from $Y$ to $X$
and $\mc N(Y, X)$ denote the space of nuclear operators from $Y$ to
$X$. Then,
for any reasonable crossnorm $\|\cdot\|$
we get the inclusions
(see \cite[Corollary 4.84]{HB})
\begin{align*}
    \mc N(Y, X)
    \cong
    X\otimes_{\|\cdot\|_{\wedge}} Y^*
    \subset
    X\otimes_{\|\cdot\|} Y^*
    \subset
    X\otimes_{\|\cdot\|_{\vee}} Y^*
    \cong
    \clos{\|\cdot\|_{X\leftarrow Y}}{\mc F(Y, X)}
    \subset
    \mc K(Y, X).
\end{align*}

An important example and the subject of this work is the case when
$X$ and $Y$ are Hilbert spaces. Then, $X^*\cong X$ and $Y^*\cong Y$.
This implies that if $\|\cdot\|$ is a reasonable crossnorm, then
$
    X\otimes_{\|\cdot\|} Y
    \subset
    \mc K(Y, X).
$
Since we can apply the singular value decomposition in
$\mc K(Y, X)$, this gives a representation for any
$v\in X\otimes_{\|\cdot\|} Y$,
$
    v=\sum_{k=1}^\infty\sigma_k \psi_k\otimes \phi_k,
$
for a decreasing non-negative sequence $\{\sigma_k\}_{k\in\N}$ and orthonormal
systems
$\{\psi_k\}_{k\in\N}\subset X$, $\{\phi_k\}_{k\in\N}\subset Y$.
Moreover, this provides us with the best low-rank approximations
\begin{align*}
    \|v-\sum_{k=1}^r\sigma_k x_k\otimes y_k\|_{\vee}=
    \sigma_{r+1}=\inf_{v_r=\sum_{k=1}^rx_k\otimes y_k}\|v-v_r\|_{\vee},
\end{align*}
for any $r\in\N$. The rank of $v$ is the smallest $r$ such that
$\sigma_{r+1}=0$ ($r=\infty$ if no such $r$ exists).
For the canonical norm on Hilbert tensor spaces we get
\begin{align*}
    \|v-\sum_{k=1}^r\sigma_k x_k\otimes y_k\|^2=
    \sum_{k=r+1}^\infty(\sigma_k)^2=
    \inf_{v_r=\sum_{k=1}^rx_k\otimes y_k}\|v-v_r\|^2.
\end{align*}
The corresponding space is the space of Hilbert Schmidt operators from
$Y$ to $X$,
$
    X\otimes_{\|\cdot\|} Y\cong\text{HS}(Y, X).
$
A typical example is the space of square integrable functions
$
    L^2(\Omega)\cong
    L^2(\Omega_1)\otimes_{\|\cdot\|_0} L^2(\Omega_2),
$
where we consider product domains
$\Omega=\Omega_1\times\Omega_2$.
The Sobolev space $H^1(\Omega)$, on the other hand,
is not equipped with the canonical norm. The Hilbert tensor space
that is the tensor product of one dimensional Sobolev spaces
with the canonical norm corresponds
to the space $H^1_{\text{mix}}(\Omega)$ of functions with mixed smoothness.
The singular value
decomposition does not apply in $H^1$ directly, which is the motivation for
this work.

The above does not extend to $d>2$ directly. However, we have
the following vector space isomorphism.

\begin{definition}[Matricisation]\label{eq:iso}
The matricisation $\mc M_\alpha$ with $\emptyset\neq\alpha\subsetneq\{
1,\ldots,d\}$ is the linear map defined by
    \begin{align*}
        \mc M_\alpha:\;&{}_a\bigotimes_{j=1}^d X_j\rightarrow
        \left({}_a\bigotimes_{j\in\alpha}X_j\right)\otimes_a
        \left({}_a\bigotimes_{j\in\alpha^c}X_j\right),\\
        &\bigotimes_{j=1}^dx_j\mapsto
        \left(\bigotimes_{j\in\alpha}x_j\right)\otimes
        \left(\bigotimes_{j\in\alpha^c}x_j\right),
    \end{align*}
    where the definition can be extended to any $x\in
    {}_a\bigotimes_{j=1}^d X_j $ by linearity. Moreover,
    this definition can be extended to topological tensors, if the
    norms in the domain and image of $\mc M_\alpha$ are compatible, i.e.,
    $\mc M_\alpha$ and $\mc M^{-1}_\alpha$ are continuous.
\end{definition}

\begin{definition}[HOSVD]\label{def:hosvd}
    Let $\alpha$ be as in \eqref{eq:iso} and let\\
    $V= {}_{\|\cdot\|}\bigotimes_{j=1}^d X_j$, where all $X_j$ are Hilbert
    spaces and $\|\cdot\|$ is the canonical norm. 
    Then, $\mc M_\alpha$ is an linear isometric isomorphism from $V$
    to the Hilbert tensor space
$
        \left({}_a\bigotimes_{j\in\alpha}X_j\right)\otimes_{\|\cdot\|_\alpha}
        \left({}_a\bigotimes_{i\in\alpha^c}X_i\right),
$
    endowed with the canonical norm $\|\cdot\|_\alpha$. Thus, we can apply
    SVD for any $\alpha$. Set $\alpha=\{j\}$ and let
    $\{\psi^j_k\}_{k\in\N}\subset X_j$
    denote the $X_j$-orthonormal singular functions obtained from the SVD of
    $\mc M_{\{j\}}(x)$. Then, there exists a unique sequence $\bs x\in\ell_2(\N^d)$
    such that
$
        x=\sum_{k_1,\ldots,k_d=1}^\infty \bs x_{k_1,\ldots,k_d}
        \psi_{k_1}^1\otimes\cdots\otimes\psi_{k_d}^d.
$
    This representation is called
    the higher-order singular value decomposition\\
    (HOSVD) of $x$. 
    \end{definition}
Other types of decompositions can be obtained by considering SVDs of\\
$\alpha$-matricisations for all $\alpha$ in a dimension partition
tree over $\{1,\cdots,d\}$. These decompositions are called hierarchical HOSVDs.
For details and precise definitions see \cite[Sections 8.3, 11.3]{HB}.

The approximation obtained by truncating the HOSVD is not optimal anymore but
rather quasi-optimal, as recalled in the following theorem. The proof can be found in
\cite[Theorem 10.3]{HB}.
\begin{theorem}[HOSVD truncation]\label{th:hosvd}
    In the setting of Definition
    \cref{def:hosvd}, let $\bs r:=(r_1,\ldots, r_d) \in \mathbb{N}^d$
    and let $P^j_{r_j}$ be the orthogonal
    projection from $X_j$ onto
$
        U^j_{r_j}(x)=\linspan{\psi^j_k:1\leq k\leq r_j},
$
    where $\psi_k^j\in X_j$ are the singular functions obtained via the HOSVD. 
    Then $x_{\bs r} := \mc P_{\bs r} x$, with 
   $ \mc P_{\bs r} =  \bigotimes_{j=1}^d P_{r_j}^j x$, is called the truncated HOSVD with multilinear (Tucker) rank 
    $\bs r$, and the  truncation error satisfies
    \begin{align*}
        \|x-x_{\bs r}\|^2\leq
        \sum_{j=1}^d\sum_{i=r_j+1}^\infty (\sigma_i^j)^2\leq d
        \inf_{v\in \mc T_{\bs r}(V)}\|x-v\|^2,
    \end{align*}
    where $\{\sigma_i^j\}_{i\in\N}$ are the singular values of $\mc M_{\{j\}}(x).$
\end{theorem}
Similar statements can be obtained for the hierarchical HOSVD, with different
constants \cite{Lars, HBSEMA}.

\section{Low-Rank Approximations in $H^1$}\label{sec:approx}
Before we continue with our analysis of low-rank approximations, we clarify
what is meant by an algebraic tensor in $H^1(\Omega)$. So far we defined
algebraic tensors only on tensor product spaces. In the case
of intersection spaces there are several candidates, since there are
multiple tensor product spaces involved.
As the following lemma shows, all possible choices lead to the same
algebraic tensor space, as long as we require $H^1$ regularity.

\begin{lemma}[{\cite[Proposition 4.104]{HB}}]\label{lemma:tensalg}
    \begin{align*}
        {}_a\bigotimes_{j=1}^d L^2(\Omega_j)\bigcap
        H^1(\Omega)= {}_a\bigotimes_{j=1}^d H^1(\Omega_j).
    \end{align*}
\end{lemma}
\begin{proof}
    To show how the algebraic tensors in $L^2(\Omega)$ inherit $H^1$
    regularity, we detail the proof in a more rigorous way than in
    \cite[Proposition 4.104]{HB}.
    The inclusion ``$\supset$'' is obvious. For the inclusion ``$\subset$'',
    let 
    $u\in {}_a\bigotimes_{j=1}^dL^2(\Omega_j)\bigcap H^1(\Omega)$.
    For a fixed $1\leq k\leq d$, we have
    \begin{align}\label{eq:h1reg}
        u\in L^2(\Omega_k)\otimes_a\left({}_a\bigotimes_{j\neq k}L^2(\Omega_j)
        \right)
        \cap\clos{\|\cdot\|_{e_k}}
        {H^1(\Omega_k){}_a\bigotimes_{j\neq k}L^2(\Omega_j)}.
    \end{align}
    Then, there is a number $r\in\N$ and functions
    $\{v_l\}_{l=1}^r\subset L^2(\Omega_k)$,\\
    $\{w_l\}_{l=1}^r\subset{}_a\bigotimes_{j\neq k}L^2(\Omega_j)$
    such that
$
        u=\sum_{l=1}^r v_l\otimes w_l.
$
    By \cite[Lemma 3.13]{HB}, w.l.o.g., we can assume $\{v_l\}_{l=1}^r$ and
    $\{w_l\}_{l=1}^r$ to be linearly independent. Thus, we can choose a dual
    basis $\{\varphi_l\}_{l=1}^r\subset
    \left({}_a\bigotimes_{j\neq k}L^2(\Omega_j)\right)^*$ such that
    $\varphi_k(w_l)=\delta_{kl}$.
    Note that the mapping
    \[
        \id_k\otimes\varphi_l:
        H^1(\Omega_k)\otimes_{\|\cdot\|_{e_k}}\left({}_a\bigotimes_{j\neq k}L^2(\Omega_j)
        \right)\rightarrow
        H^1(\Omega_k)
    \]
    is continuous for all $1\leq l\leq r$, since $\|\cdot\|_{e_k}$ is a
    reasonable crossnorm.

    Moreover, by \eqref{eq:h1reg},
    there exist
    $\{v_l^m\otimes w_l^m\}_{\substack{1\leq l\leq m,\\
    1\leq m<\infty}}\subset
    H^1(\Omega_k)\otimes_a\left({}_a\bigotimes_{j\neq k}L^2(\Omega_j)\right)$
    such that
    \[
        \lim_{m\rightarrow\infty}
        \left\|u-\sum_{l=1}^mv_l^m\otimes w_l^m\right\|_{e_k}=0.
    \]
    Thus, since $\id_k\otimes\varphi_i$ is continuous
    \begin{align*}
        0&=\lim_{m\rightarrow\infty}\left\|\id_k\otimes\varphi_i
        \left(u-\sum_{l=1}^mv_l^m\otimes w_l^m\right)\right\|_1\\
        &=\lim_{m\rightarrow\infty}\left\|\sum_{l=1}^rv_l\varphi_i(w_l)-
        \sum_{l=1}^mv_l^m\varphi_i(w_l^m)\right\|_1\\
        &=\lim_{m\rightarrow\infty}\left\|v_i-
        \sum_{l=1}^mv_l^m\varphi_i(w_l^m)\right\|_1.
    \end{align*}
    And thus
$
        v_i\in\clos{\|\cdot\|_1}{\linspan{v_l^m:1\leq l\leq m,\;
        1\leq m<\infty}}\subset H^1(\Omega_k).
$
    Since $i$ and $k$ were chosen arbitrarily, this shows
$
        u\in \bigcap_{k=1}^d
        H^1(\Omega_k)\otimes_a\left({}_a\bigotimes_{j\neq k}L^2(\Omega_j)
        \right).
$
    Finally, by \cite[Lemma 6.11]{HB}
$
        \bigcap_{k=1}^d
        H^1(\Omega_k)\otimes_a\left({}_a\bigotimes_{j\neq k}L^2(\Omega_j)
        \right)
        ={}_a\bigotimes_{j=1}^d H^1(\Omega_j).
$
    This completes the proof.
\end{proof}

\subsection{Existence of Low-Rank Approximations}
First, we address the question of existence of low-rank approximations
of a function $u\in H^1(\Omega)$. Since for $d>2$ and $r>1$
the set $\mc R_r(V)$ is not closed even for the case
$V=L^2(\Omega)$ (see \cite[Section 9.4.1]{HB}), we only consider
Tucker formats.

In analogy to
Definition \cref{def:minsub}, for $u\in H^1(\Omega)$ we define the subspace
\begin{align}\label{eq:subsph1}
    U^j(u) := \linspan{\varphi(u):\varphi=\bigotimes_{k=1}^d\varphi_k,\;
    \varphi_j=\id_j,\;\varphi_k\in\left(L^2(\Omega_k))^*,\;k\neq j\right)}.
\end{align}
Note that for $u\in {}_a\bigotimes_{j=1}^d H^1(\Omega_j)$, $U^j(u)$
is a closed
subspace of $H^1(\Omega_j)$ and the definition coincides with the case
$u\in {}_a\bigotimes_{j=1}^d L^2(\Omega_j)$. Since in this case $u$ can be
written as
$
    u=\sum_{k=1}^r\bigotimes_{i=1}^dv_k^i,
$
for some $r\in\N$, any $\varphi$ from \eqref{eq:subsph1} applied to $u$ yields
$
    \varphi(u)=\sum_{k=1}^rv_k^j\left(\prod_{i\neq j}\varphi_i(v_k^i)\right).
$
And thus
$
   U^j(u)\subset\linspan{v_k^j:1\leq k\leq r}.
$
The subspace $U^j(u)\subset H^1(\Omega_j)$ is finite dimensional and is
closed in any norm.

Thus, together with Lemma \cref{lemma:tensalg}, we can define the Tucker manifold for
$\r=(r_1,\ldots,r_d)$ and $V={}_a\bigotimes_{j=1}^d H^1(\Omega_j)$ in the same
way as in Definition \cref{def:formatsTucker}. This set remains weakly
closed in $H^1(\Omega)$. To show this, we first require the following lemma.
\begin{lemma}\label{lemma:tuckermanifold}
    $\T\left({}_a\bigotimes_{j=1}^d H^1(\Omega_j)\right)
    ={}_a\bigotimes_{j=1}^d H^1(\Omega_j)\cap
    \T\left({}_a\bigotimes_{j=1}^d L^2(\Omega_j)\right)$.
\end{lemma}
\begin{proof}
    The inclusion ``$\subset$'' is trivial.

    For the other inclusion, assume $v\in
    {}_a\bigotimes_{j=1}^d H^1(\Omega_j)\cap
    \T\left({}_a\bigotimes_{j=1}^d L^2(\Omega_j)\right)$. Since
    $v\in {}_a\bigotimes_{j=1}^d H^1(\Omega_j)$, $U^j(v)\subset H^1(\Omega_j)$
    and by \cite[Lemma 6.11]{HB}
    $v\in {}_a\bigotimes_{j=1}^d U^j(v)\subset
    {}_a\bigotimes_{j=1}^d L^2(\Omega_j)$.
    In particular, since $v\in
    \T\left({}_a\bigotimes_{j=1}^d L^2(\Omega_j)\right)$,
    $\dim U^j(v)\leq r_j$ for all $1\leq j\leq d$. Hence,
    $v\in\T\left({}_a\bigotimes_{j=1}^d H^1(\Omega_j)\right)$.
    This completes the proof.
\end{proof}

\begin{theorem}
    $\T\left({}_a\bigotimes_{j=1}^d H^1(\Omega_j)\right)$ is
    weakly closed and therefore proximinal in $H^1(\Omega)$.
\end{theorem}
\begin{proof}
    Let $\{v_n\}_{n\in\N}\subset
    \T\left({}_a\bigotimes_{j=1}^d H^1(\Omega_j)\right)$ satisfy
    $v_n\rightharpoonup v\text{ in }H^1(\Omega)$.
    Since $\left(L^2(\Omega)\right)^*\subset \left(H^1(\Omega)\right)^*$,
    $v_n\rightharpoonup v\text{ in }L^2(\Omega)$.
    By Lemma \cref{lemma:tuckermanifold},
    $\{v_n\}_{n\in\N}\subset
    \T\left({}_a\bigotimes_{j=1}^d L^2(\Omega_j)\right)$, i.e.,
    $\dim U^j(v_n)\leq r_j$ for all $1\leq j\leq d$.
    By \cite[Theorem 6.24]{HB},
    $\dim U^j(v)\leq\liminf_{n\rightarrow\infty}\dim U^j(v)\leq r_j$,
    and thus
    $v\in\T\left({}_a\bigotimes_{j=1}^d H^1(\Omega_j)\right)$.
    Since $H^1(\Omega)$ is a reflexive Banach space, the set
    $\T\left({}_a\bigotimes_{j=1}^d H^1(\Omega_j)\right)$ is proximinal.
\end{proof}

\subsection{Minimal Subspaces}
The subspaces from \eqref{eq:subsph1} inherit $H^1$ regularity.
\begin{lemma}\label{lemma:h1reg}
    For $u\in H^1(\Omega)$,
$
        U^j(u)\subset H^1(\Omega_j).
$
\end{lemma}
\begin{proof}
    Let $v\in {}_a\bigotimes_{j=1}^d H^1(\Omega_j)$ and
$
        \varphi_{[j]}:=\bigotimes_{k=1}^d\varphi_k,$ $
        \varphi_j=\id_j,$ $
        \varphi_k\in L^2(\Omega_k)^*,\;k\neq j.
$
    Clearly, $\varphi_{[j]}(v)\in H^1(\Omega_j)$. By \cite[Lemma 4.97]{HB}
    and \cite[Proposition 4.68]{HB}
    \begin{align*}
        \|\varphi_{[j]}(v)\|_1\lesssim\|v\|_{\vee(e_j)}\lesssim
        \|v\|_{e_j}.
    \end{align*}
    Thus, $\varphi_{[j]}: \left({}_a\bigotimes_{j=1}^d H^1(\Omega_j),
    \|\cdot\|_1\right)
    \rightarrow H^1(\Omega_j)$ is a continuous linear mapping. Since\\
    ${}_a\bigotimes_{j=1}^d H^1(\Omega_j)$ is dense in $H^1(\Omega)$,
    $\varphi_{[j]}$ can be uniquely extended to a bounded linear mapping
    on $H^1(\Omega)$ with the same operator norm, i.e.,
    $\varphi_{[j]}(v)\in H^1(\Omega_j)$ is well defined for
    $v\in H^1(\Omega)$ and the statement follows.
\end{proof}

Before we proceed, we would like to clarify that there are several possible
definitions for minimal subspaces when considering $u\in H^1(\Omega)$.
First, there are two possible choices for the dual space leading to
\begin{align*}
    U_{a}^j(u)&:= \linspan{\varphi(u):\varphi=\bigotimes_{k=1}^d \varphi_k,\;
    \varphi_j=\id,\;\varphi_k\in\left(L^2(\Omega_k))^*,\;k\neq j\right)},\\
    U_{b}^j(u)&:= \linspan{\varphi(u):\varphi=\bigotimes_{k=1}^d \varphi_k,\;
    \varphi_j=\id,\;\varphi_k\in\left(H^1(\Omega_k))^*,\;k\neq j\right)},
\end{align*}
where clearly $U_a^j(u)=U^j(u)$.
Second, there are two possible choices for the completion norm, which overall
leads to four possible definitions
\begin{align*}
    U_{I}^j(u)&:=\clos{\|\cdot\|_0}{U^j_{a}(u)},\;
    U_{II}^j(u):=\clos{\|\cdot\|_1}{U^j_{a}(u)},\\
    U_{III}^j(u)&:=\clos{\|\cdot\|_0}{U^j_{b}(u)},\;
    U_{IV}^j(u):=\clos{\|\cdot\|_1}{U^j_{b}(u)}.
\end{align*}
The space $U_{I}^j(u)$ is the minimal subspace of $u$ as a function in
$L^2(\Omega)$.
For $d=2$,
$U_{II}^1(u)$ (resp. $U_{II}^2(u)$) is the minimal subspace
of $u$ as a function in $H^{(1,0)}$ (resp. $H^{(0,1)}$),
$U_{III}^1(u)$ (resp. $U_{III}^2(u)$) is the minimal subspace of
$u$ as a function in $H^{(0,1)}$ (resp. $H^{(1,0)}$) and
$U_{IV}^j(u)$ is the minimal subspace of $u$ as a function in
$H^1_\mix(\Omega)$.
Since we want to consider precisely $u\in H^1(\Omega)$, we consider
$u\in H^{(1,0)}$ and choose the
variant $U_{II}^1(u)$ for the left minimal subspace, and
$u\in H^{(0,1)}$ with the variant $U_{II}^2(u)$ for the right minimal subspace.
Analogously for $d>2$.

With the preceding lemma we may now define
$
    \Umin_j(u)=\clos{\|\cdot\|_1}{U_a^j(u)}\subset H^1(\Omega_j).
$
This space differs from $U^j(u)$
in case $u\not\in {}_a\bigotimes_{j=1}^d H^1(\Omega_j)$.
We want to check if property \eqref{eq:inmin} still holds for
$H^1$ functions. To this end, we require the following assumption.

\begin{assumption}\label{ass:pl2}
    Let $\mc P^j:H^1(\Omega_j)\rightarrow \Umin_j(u)$ be an orthogonal
    projection. We assume $\mc P^j$ is continuous in $L^2$
    \begin{align*}
        \sup_{\substack{v_j\in H^1(\Omega_j),\\v\neq 0}}\frac{\|\mc P^jv_j\|_0}
        {\|v\|_0}<\infty.
    \end{align*}
\end{assumption}

\begin{remark}
    We will frequently encounter Assumption \cref{ass:pl2} in the following
    sections. We will discuss sufficient conditions for this assumption
    to be satisfied in part II of this series.
    We will see that this assumption is not necessarily satisfied.
    In fact, we conjecture that there are functions $u\in H^1(\Omega)$ which
    do not satisfy the statement of Theorem \cref{thm:minsubs}.
    The proof of this, however, seems to be not trivial.
\end{remark}

\begin{proposition}\label{thm:minsubs}
    Let $u\in H^1(\Omega)$ and assume \cref{ass:pl2} is satisfied.
    Then, it holds
$
        u\in{}_{\|\cdot\|_1}\bigotimes_{j=1}^d\Umin_j(u).
$
\end{proposition}
\begin{proof}
    Since $u\in H^1(\Omega)$, by \eqref{eq:sobolev}, $u\in
    H^1(\Omega_j)\otimes_{\|\cdot\|_{e_j}}
    \left({}_a\bigotimes_{k\neq j}L^2(\Omega_k)\right)$. The space
    $\Umin_j(u)\subset H^1(\Omega_j)$ is a closed subspace of the
    Hilbert space $H^1(\Omega_j)$. Thus, $\Umin_j(u)\in\G(H^1(\Omega_j))$.
    Moreover, $\|\cdot\|_{e_j}$ is a uniform crossnorm. This holds for
    any $1\leq j\leq d$ and thus by
    \cite[Theorem 6.29]{HB} we obtain
    \begin{align*}
        u\in\bigcap_{k=1}^d
        \Umin_k(u)\otimes_{\|\cdot\|_{e_k}}
        \left({}_a\bigotimes_{j\neq k}L^2(\Omega_j)\right).
    \end{align*}

    Next, following the arguments of \cite[Theorem 6.28]{HB},
    consider the orthogonal projection
    $P^j:H^1(\Omega_j)\rightarrow \Umin_j(u)$. Let
    $\mc P^{j}:=P^j\otimes
    \left(\bigotimes_{k\neq j}^d \mc\id_k\right)$.
    This is a
    linear continuous mapping from
    $H^1(\Omega_j)
    \otimes_{\|\cdot\|_{e_j}}\left({}_a\bigotimes_{k\neq j}L_2(\Omega_k)
    \right)$
    to
    $\Umin_j(u)\otimes_{\|\cdot\|_{e_j}}\left(
    {}_a\bigotimes_{k\neq j}L^2(\Omega_k)\right)$,
    with $\|\mc P^{j}\|=\|P^j\|=1$ (since $\|\cdot\|_{e_j}$ is a
    uniform crossnorm).

    Take a sequence $\{u_n\}_{n\in\N}\subset
    \Umin_j(u)\otimes_a\left({}_a\bigotimes_{k\neq j}L^2(\Omega_k)
    \right)$ such that
$
        \lim_{n\rightarrow\infty}\|u-u_n\|_{e_j}=0.
$
    Clearly, $\mc P^{j}(u_n)=u_n$ and
    \begin{align*}
        \|u-\mc P^{j}u\|_{e_j}&\leq
        \|u-u_n\|_{e_j}+
        \|u_n-\mc P^{j}u\|_{e_j}=
        \|u-u_n\|_{e_j}+
        \|\mc P^{j}(u_n-u)\|_{e_j}\\
        &\leq 2
        \|u-u_n\|_{e_j}.
    \end{align*}
    Taking the limit with $n$, we obtain $u=\mc P^{j}u$.

    Since this holds for any $1\leq j\leq d$, we get
    \begin{align}\label{eq:uid}
        u=\left(\prod_{j=1}^d \mc P^{j}\right)u
        =\left(\bigotimes_{j=1}^d P^j\right)u.
    \end{align}

    Next, we require a separable representation for $u$ that converges in
    $H^1(\Omega)$. This is possible for $H^1(\Omega)$ by choosing a
    complete $H^1(\Omega)$-orthonormal system of elementary tensor products
    (e.g., a Fourier basis) or an $H^1(\Omega)$ Riesz basis of wavelets.
    Let $\{\bigotimes_{j=1}^d\psi^j_{k_j}:(k_j)_{j=1}^d\in\N^d\}$ be
    such a system.
    Then, there exists a sequence
    $\bs u=(\bs u_{\bs k})_{\bs k\in\N^d}\in\ell_2(\N^d)$ such
    that
    \begin{align}\label{eq:useries}
        u=\lim_{n\rightarrow\infty}\sum_{k_1=1}^n\cdots\sum_{k_d=1}^n
        \bs u_{\bs k}\bigotimes_{j=1}^d\psi^j_{k_j},
    \end{align}
    with convergence in $\|\cdot\|_1$. Since
$
        \|\cdot\|_{e_j}\leq\|\cdot\|_1,\quad 1\leq j\leq d,
$
    \eqref{eq:useries} converges in $\|\cdot\|_{e_j}$ for all
    $1\leq j\leq d$ as well. Thus, by \eqref{eq:uid} and
    Assumption \cref{ass:pl2}
    \begin{align*}
        u=\left(\bigotimes_{j=1}^d P^j\right)u
        =\lim_{n\rightarrow\infty}\sum_{k_1=1}^n\cdots\sum_{k_d=1}^n
        \bs u_{\bs k}\bigotimes_{j=1}^d P^j(\psi^j_{k_j}),
    \end{align*}
    for any $1\leq i\leq d$ and
    with convergence in $H^1(\Omega)$. We take\\
$
        u_n:=\sum_{k_1=1}^n\cdots\sum_{k_d=1}^n
        \bs u_{\bs k}\bigotimes_{j=1}^d P^j(\psi^j_{k_j}).
$
    Clearly, $u_n\in{}_a\bigotimes_{j=1}^d\Umin_j(u)$ and by above
$
        \lim_{n\rightarrow\infty}\|u-u_n\|_1=0.
$
    This completes the proof.
\end{proof}

\section{$H^1$-Error Analysis of $L^2$-SVD}\label{sec:l2svd}
The singular value decomposition can be utilized to obtain spaces
$\Umin_j(u)$ and low-rank approximations therein. Interestingly, the resulting
spaces are not necessarily the same depending on the interpretation of
$u\in H^1(\Omega).$ In the following we restrict the exposition to $d=2$.

If we consider $u\in L^2(\Omega)$, then $u$
can be identified with a compact operator from
$L^2(\Omega_2)$ to $L^2(\Omega_1)$ such that
\begin{align}\label{eq:hsl2}
    u[w] = \int_{\Omega_2}u(\cdot, y)w(y)dy.
\end{align}
for $w\in L^2(\Omega_2)$. The adjoint $u^*:L^2(\Omega_1)\rightarrow
L^2(\Omega_2)$ is given by
$
    u^*[v] = \int_{\Omega_1}u(x, \cdot)v(x)dx,
$
for $v\in L^2(\Omega_1)$. Thus, a left singular vector $\psi$ of $u$ satisfies
$
    uu^*[\psi] = \int_{\Omega_2}u(\cdot, y)\int_{\Omega_1}u(x,y)\psi(x)dxdy
    = \lambda\psi,
$
for some $\lambda\in\R^+$, and the accompanying right singular vector satisfies
$
    u^*u[\phi] = \int_{\Omega_1}u(x, \cdot)\int_{\Omega_2}u(x,y)\phi(y)dydx
    = \lambda\phi.
$

Since $u$ is a compact operator, we can find an
$L^2$-orthonormal system of left and right singular vectors, which we denote by
$\{\psi_k\}_{k\in\N}$ and $\{\phi_k\}_{k\in\N}$, respectively, and the
corresponding singular values by $\{\sigma_k^{00}=\sqrt{\lambda^{00}_k}\}_{k\in\N}$,
sorted by decreasing values such that
$
    u=\lim_{r\rightarrow\infty}\sum_{k=1}^r\sigma_k^{00}\psi_k\otimes\phi_k,
        $ $\text{in }\|\cdot\|_0.
$
We have the identities
$
    U_1(u)=\linspan{\psi_k:k\in\N,\;\sigma^{00}_k>0},$\\$
    U_2(u)=\linspan{\phi_k:k\in\N,\;\sigma^{00}_k>0}.
$
The SVD provides both optimal low-rank approximations of given rank
and an error estimator in the sense that
\begin{align*}
    \|u-\sum_{k=1}^r\sigma_k^{00}\psi_k\otimes\phi_k\|_0
    &=\inf_{g\in\mc R_r(L^2(\Omega))}
    \|u-g\|_0,\\
    \|u-\sum_{k=1}^r\sigma_k^{00}\psi_k\otimes\phi_k\|^2_0
    &=\sum_{k=r+1}^\infty(\sigma_k^{00})^2.
\end{align*}
For the case $u\not\in L^2(\Omega_1)\otimes_a L^2(\Omega_2)$, $\lambda^{00}_k>0$
for all $k\in\N$. Otherwise, we only require finitely many
$\psi_k$'s and $\phi_k$'s. Letting\\
$
    \gamma_k= \frac{1}{\lambda_k^{00}}
    \int_{\Omega_2}\frac{\partial}
    {\partial x}u(\cdot, y)\int_{\Omega_1}u(x,y)\psi_k(x)dxdy,
$
we have that
\begin{align}\label{eq:dpsi}
   &\left\|\gamma_k\right\|_0^2=\frac{1}{(\lambda_k^{00})^2}
    \int_{\Omega_1}\left(\int_{\Omega_2}\frac{\partial}{\partial x}u(s, y)
    \int_{\Omega_1}u(x, y)\psi_k(x)dxdy
    \right)^2ds\\
    &\leq\frac{1}{(\lambda_k^{00})^2}
    \int_{\Omega_1}\left(\int_{\Omega_2}\frac{\partial}{\partial x}u(s, y)
    \left(\int_{\Omega_1}u^2(x, y)dx\right)^{1/2}
    \left(\int_{\Omega_1}\psi^2_k(x)dx\right)^{1/2}dy
    \right)^2ds\notag\\
    &\leq\frac{1}{(\lambda_k^{00})^2}\|\psi_k\|^2_0
    \int_{\Omega_1}\left(\left(\int_{\Omega_2}(\frac{\partial}{\partial x}u(s, y))^2dy
    \right)^{1/2}
    \left(\int_{\Omega_2}\int_{\Omega_1}u^2(x, y)dxdy\right)^{1/2}
    \right)^2ds\notag\\
    &=\frac{1}{(\lambda_k^{00})^2}\|u\|_0^2\left\|
    \frac{\partial}{\partial x}u\right\|_0^2,\notag
\end{align}
and for all $\varphi\in C_c^\infty(\Omega_1)$
\begin{align}\label{eq:weakdiff}
    \int_{\Omega_1}\varphi(s)\gamma_k(s)&=
    \int_{\Omega_1}\varphi(s)\frac{1}{\lambda_k^{00}}\int_{\Omega_2}
    \frac{\partial}{\partial x}u(s, y)\int_{\Omega_1}u(x, y)\varphi_k(x) dx dy ds
    \notag
    \\
    &=\frac{1}{\lambda_k^{00}}\int_{\Omega_2}\int_{\Omega_1}\varphi(s)
    \frac{\partial}{\partial s}u(s, y)ds\int_{\Omega_1}u(x, y)\varphi_k(x)dx dy
    \notag\\
    &=-\frac{1}{\lambda_k^{00}}\int_{\Omega_2}\int_{\Omega_1}\frac{d}{ds}\varphi(s)
    u(s, y)ds\int_{\Omega_1}u(x, y)\varphi_k(x)dx dy\notag\\
    &=-\int_{\Omega_1}\frac{d}{ds}\varphi(s)\frac{1}{\lambda_k^{00}}\int_{\Omega_2}
    u(s, y)\int_{\Omega_1}u(x, y)\varphi_k(x)dx dy ds\notag\\
    &=-\int_{\Omega_1}\frac{d}{ds}\varphi(s)\psi_k(s) ds,
\end{align}
so that $\gamma_k=\frac{d}{dx}\psi_k$.
Analogously for $\phi_k$,
\begin{align*}
    \frac{d}{dx}\phi_k= \frac{1}{\lambda_k^{00}}
    \int_{\Omega_1}\frac{\partial}{\partial y}u(x, \cdot)
    \int_{\Omega_2}u(x,y)\phi_k(y)dydx.
\end{align*}
Thus, $\psi_k\in H^1(\Omega_1)$,
$\phi_k\in H^1(\Omega_2)$ for all
$k\in\N$ and, consistently with Lemma \cref{lemma:h1reg},
$U_1(u)\subset H^1(\Omega_1)$ and
$U_2(u)\subset H^1(\Omega_2)$. The best rank $r$ approximation in $L^2$,
\begin{align}\label{eq:bestrankl2}
    u_r:=\sum_{k=1}^r\sigma_k^{00}\psi_k\otimes\phi_k,
\end{align}
makes sense in $H^1$ and we can consider the error
$\|u-u_r\|_1$.

\begin{theorem}\label{thm:h1errorl2svd}
    Let $u\in H^1(\Omega)$ and $u_r$ be its best rank $r$ approximation in
    $L^2$ defined by \eqref{eq:bestrankl2}. We have
    \begin{align}\label{eq:urrep}
        \|u_r\|_1^2=\sum_{k=1}^r(\sigma_k^{00})^2
        \left(1+\left\|\frac{d}{dx}\psi_k\right\|^2_0+
        \left\|\frac{d}{dy}\phi_k\right\|^2_0\right)
    \end{align}
    If
    \begin{align}\label{eq:convh1}
        \lim_{r\rightarrow\infty}\|u_r\|_1<\infty
    \end{align}
    then $\|u-u_r\|_1\rightarrow 0$ and
    \begin{align}\label{eq:L2svderror}
        \|u-u_r\|^2_1=\sum_{k=r+1}^\infty(\sigma_k^{00})^2
        \left(1+\left\|\frac{d}{dx}\psi_k\right\|^2_0+
        \left\|\frac{d}{dy}\phi_k\right\|^2_0\right).
    \end{align}
\end{theorem}
\begin{proof}
    Clearly, if $u_r$ converges to some $\tilde{u}\in H^1(\Omega)$,
    $u=\tilde{u}$ a.e.\ by the
    simple inequality
    \begin{align}\label{eq:unique}
        \|u-\tilde{u}\|_0\leq\inf_{r\in\N}\left\{
        \|u-u_r\|_0+\|\tilde{u}-u_r\|_1\right\}=0.
    \end{align}
    We have
    \begin{align*}
        \|u_r\|_1^2&=\left\|\sum_{k=1}^r\sigma_k^{00}
        \psi_k\otimes\phi_k\right\|_1^2
        =\sum_{k,l=1}^r\sigma_k^{00}\sigma_l^{00}
        \dpair{\psi_k\otimes\phi_k}{\psi_l\otimes\phi_l}_1\\
        &=\sum_{k,l=1}^r\sigma_k^{00}\sigma_l^{00}
        \bigg(\dpair{\psi_k\otimes\phi_k}{\psi_l\otimes\phi_l}_0
        +
        \dpair{\frac{d}{dx}\psi_k\otimes\phi_k}
        {\frac{d}{dx}\psi_l\otimes\phi_l}_0\\
        &+
        \dpair{\psi_k\otimes\frac{d}{dy}\phi_k}
        {\psi_l\otimes\frac{d}{dy}\phi_l}_0
        \bigg)\\
        &=
        \sum_{k,l=1}^r\sigma_k^{00}\sigma_l^{00}
        \left(\delta_{kl}\delta_{kl}+\left\|\frac{d}{dx}\psi_k\right\|_0^2
        \delta_{kl}+\delta_{kl}\left\|\frac{d}{dy}\phi_k\right\|_0^2
        \right)\\
        &=\sum_{k=1}^r(\sigma_k^{00})^2
        \left(1+\left\|\frac{d}{dx}\psi_k\right\|_0^2
        +\left\|\frac{d}{dy}\phi_k\right\|_0^2
        \right).
    \end{align*}
    Thus, $(\|u_r\|_1^2)_{r\in\N}$ is a positive increasing sequence. If
    \eqref{eq:convh1} holds, then
    $(\|u_r\|_1^2)_{r\in\N}$ converges. Then,
    for $m\geq r$
$
        \|u_m-u_r\|_1^2=\|u_m\|_1^2-\|u_r\|_1^2,
$
    which proves that $u_r$ is Cauchy and therefore converges. Taking the
    limit and by \eqref{eq:unique}, we obtain $\|u-u_r\|_1\rightarrow 0$.
    The proof of \eqref{eq:L2svderror} follows similarly as above.
\end{proof}

\begin{remark}\label{rem:multiscale}
Equation \eqref{eq:L2svderror} is thus a recipe for constructing low-rank
approximations via the $L^2$-SVD but with error control in $H^1$. Assumption
\eqref{eq:convh1}
particularly holds when $u$ is a numerical approximation to the
solution of a PDE.

We can not expect \eqref{eq:convh1} to hold in general. Specifically,
in \eqref{eq:dpsi} we applied twice the Cauchy-Schwarz inequality, which is known
to be sharp. Since $\lambda_k^{00}=(\sigma^{00}_k)^2$, this would imply that
\eqref{eq:convh1}
is not satisfied and $u_r$ diverges in $H^1(\Omega)$.
On the other hand, we can think of cases where \eqref{eq:convh1} is satisfied,
such as in the case of a Fourier basis.

We ask
what are the possible conditions on
$\psi_k$ and $\phi_k$ for \eqref{eq:convh1} to be satisfied? Note that this
condition is similar to well-known estimates from approximation theory,
specifically approximation via wavelets or, more generally, multi-scale
approximation. There sufficient conditions include the existence of a uniformly
bounded family of projectors that satisfy direct and inverse inequalities.

Translated into our setting, sufficient conditions for \eqref{eq:convh1} look
as follows. Define the subspaces
$
    S_l:=\linspan{\psi_k:1\leq k\leq 2^l}\subset \Umin_1(u),$ $l\in\N_0.
$
We require the \emph{Jackson (direct)} inequality to be satisfied
\begin{align*}
    \inf_{v_l\in S_l}\left\|f-v_l\right\|_0\lesssim 2^{-sl}\|f\|_1,
    \quad\forall f\in\Umin_1(u),
\end{align*}
for some $s>1$ and the \emph{Bernstein (indirect)} inequality
$\|v_l\|_1\lesssim 2^{\bar{s}l}\|v_l\|_0$, for all $v_l\in S_l$,
for some $\bar{s}>1$. Analogously for the space generated by the $\phi_k$'s.
For more details we refer to \cite[Theorem 5.12]{KU} and
\cite{DahmenRemarks}.
\end{remark}

We conclude this part by extending the result to $d\geq 2$. Let
$u\in H^1(\Omega)$ with $\Omega=\bigtimes_{j=1}^d\Omega_j$,
$x=(x_1,\ldots,x_d)\in\Omega$.
Then,
for any $1\leq i\leq d$, we can consider the integral operator
\begin{align*}
    u_i:L^2(\Omega_i)\rightarrow L^2(\bigtimes_{j\neq i}\Omega_j),\quad
    u_i[w]=\int_{\Omega_i}u(\cdot, \ldots, x_i,\ldots, \cdot)w(x_i)dx_i,
    w\in L^2(\Omega_i).
\end{align*}
As before, we can consider the singular vectors
$\{\psi_k^i:k\in\N\}$, and the corresponding eigenvalues
$\{\lambda_k^i\in\R^+:k\in\N\}$. The
derivatives are given by
\begin{align*}
    \frac{d}{dx_i}\psi_k^i=
    \int_{\bigtimes_{j\neq i}\Omega_j}
    \frac{\partial}{\partial x_i}
    u(\ldots,x_{i-1},\cdot,x_{i+1},\ldots)\int_{\Omega_i}
    u(x)\psi^i_k(x_i)dx.
\end{align*}
with the familiar estimate
$
    \left\|\frac{d}{dx_i}\psi_k^i\right\|_0\leq
    \frac{1}{\lambda_k^i}\|u\|_0\left\|\frac{\partial}{\partial x_i}u
    \right\|_0.
$
The identity
\begin{align*}
    U_i(u)=\linspan{\psi_k^i:k\in\N,\;\sigma_k^i=\sqrt{\lambda_k^i}>0},
    \quad 1\leq i\leq d,
\end{align*}
holds.
Define the subspace
$
    B^i_{r_i}:=\linspan{\psi_k^i:1\leq k\leq r_i},
$
and the corresponding $L^2$-orthogonal projector
$P^i_{r_i}:L^2(\Omega_i)\rightarrow B^i_{r_i}$, for
$1\leq i\leq d$. Then, for
$\bs r=(r_1,\ldots,r_d)\in\N^d$, define
\begin{align}\label{eq:ubsr}
    \mc P_{\bs r}&:=\bigotimes_{j=1}^d P^j_{r_j},\quad\text{and} \quad 
    u_{\bs r}:=\mc P_{\bs r}u.
\end{align}
The projection $\mc P_{\bs r}$ is the HOSVD projection from
Theorem \cref{th:hosvd}. Before we proceed, we require the following lemma, which
is an extension of Theorem \cref{thm:h1errorl2svd}.

\begin{lemma}\label{lemma:h1errormd}
    Let $u\in H^1(\Omega)$ and
$
        \mc P^{j}_{r_j}=\id_1\otimes\cdots\otimes P^j_{r_j}\otimes
        \cdots\otimes\id_d.
$
    We have
    \begin{align*}
        \|\mc P^{j}_{r_j}u\|_{e_j}^2=\sum_{k=1}^{r_j}
        (\sigma_k^j)^2\left(1+
        \left\|\frac{d}{dx}\psi_k^j\right\|_0^2\right).
    \end{align*}
    If
$
        \lim_{r_j\rightarrow\infty}\|\mc P^{j}_{r_j}u\|_{e_j}<\infty,
$
    then $\|u-\mc P^{j}_{r_j}u\|_{e_j}\rightarrow 0$ and
    \begin{align*}
        \|u-\mc P^{j}_{r_j}u\|^2_{e_j}=
        \sum_{k=r_j+1}^\infty
        (\sigma_k^j)^2\left(1+
        \left\|\frac{d}{dx}\psi_k^j\right\|_0^2\right)
    \end{align*}
\end{lemma}
\begin{proof}
    We consider the matricisation
    \begin{align*}
        \mc M_{\{j\}}(u):H^{e_j}\rightarrow
        H^1(\Omega_j)\otimes_{\|\cdot\|_{(1,0)}}
        L^2(\bigtimes_{i\neq j}\Omega_i).
    \end{align*}
    This is a linear isometric isomorphism since $\|\cdot\|_{e_j}$ and
    $\|\cdot\|_{(1,0)}$ are canonical norms (induced by the same norms).
    The space $H^1(\Omega_j)\otimes_{\|\cdot\|_{(1,0)}}
    L^2(\bigtimes_{i\neq j}\Omega_i)$ is a Hilbert tensor space of order 2
    equipped with the canonical norm, with the $H^1$ norm on the left
    and $L^2$ norm on the right. Thus, we can apply Theorem
    \cref{thm:h1errorl2svd}
    to $\mc M_{\{j\}}(u)$ and the statement follows.
\end{proof}

For the $H^1$ error we get the following result.
Remark \cref{rem:multiscale} applies here as well.

\begin{theorem}\label{thm:L2hosvd}
    Let $u\in H^1(\Omega)$ and $u_{\bs r}$ be defined by \eqref{eq:ubsr}.
    We have
    \begin{align}\label{eq:normubsr}
       \frac{1}{d}\sum_{j=1}^d\sum_{k=1}^{r_j}
        (\sigma_k^j)^2\leq
        \|u_{\bs r}\|^2_1\leq
        \sum_{j=1}^d\sum_{k=1}^{r_j}
        (\sigma_k^j)^2\left(1+\left\|\frac{d}{dx}\psi_k^j\right\|_0^2
        \right).
    \end{align}
    Define the constants
$
        \Gamma_j(r_j):=\sup_{v\in B_{r_j}^j}\frac{\|v\|_1}{\|v\|_0}.
$
    If
    \begin{align}\label{eq:limitl2hosvd}
        \lim_{\min_jr_j\rightarrow\infty}
        \sum_{j=1}^d\sum_{k=1}^{r_j}
        (\sigma_k^j)^2\left(1+\left\|\frac{d}{dx}\psi_k^j\right\|_0^2+
        \sum_{i\neq j}\Gamma^2_i(r_i)
        \right)<\infty
    \end{align}
    then $\|u-u_{\bs r}\|_1\rightarrow 0$ and
    \begin{align}\label{eq:l2hosvdbound}
        \|u-u_{\bs r}\|^2_1\sim
        \sum_{j=1}^d\sum_{k=r_j+1}^\infty
        (\sigma_k^j)^2\left(1+\left\|\frac{d}{dx}\psi_k^j\right\|_0^2
        \right).
    \end{align}
\end{theorem}
\begin{proof}
    Let $\mc P^{j}_{r_j}$ be defined as in Lemma \cref{lemma:h1errormd}.
    The projection $\prod_{i\neq j}^d\mc P^{i}_{r_i}$ is
    orthogonal in the $\|\cdot\|_{e_j}$ norm.
    The lower bound in \eqref{eq:normubsr} is an immediate consequence of
    \cite[Theorem 10.3]{HB}.
    For the upper bound we get by applying Lemma \cref{lemma:h1errormd}
    \begin{align*}
        \left\|\mc P_{\bs r}u\right\|^2_1&\leq
        \sum_{j=1}^d\left\|\mc P_{\bs r}u\right\|^2_{e_j}
        =
        \sum_{j=1}^d
        \left\|\left(\prod_{i\neq j}^d\mc P^{i}_{r_i}\right)\mc P^{j}_{r_j}
        u\right\|^2_{e_j}\\
        &\leq
        \sum_{j=1}^d
        \left\|\mc P^{j}_{r_j}
        u\right\|^2_{e_j}=
        \sum_{j=1}^d
        \sum_{k=1}^{r_j}(\sigma_k^j)^2
        \left(1+\left\|\frac{d}{dx}\psi_k^j\right\|_0^2\right).
    \end{align*}
    Next, observe that we can bound the $\|\cdot\|_1$ norm of $P_{r_j}^j$
    as follows
    \begin{align*}
        \|P_{r_j}^jv\|_1\leq\Gamma_j(r_j)\||P_{r_j}^jv\|_0
        \leq\Gamma_j(r_j)\|v\|_0\leq \Gamma_j(r_j)\|v\|_1,
    \end{align*}
    for any $v\in H^1(\Omega_j)$. Thus, since $\|\cdot\|_{e_j}$ is a uniform
    crossnorm on $H^{e_j}$,
    $\|\mc P^j_{r_j} \|_{e_j} = \| P^j_{r_j} \|_{1}$, so that we can bound
    $\|\mc P_{r_j}^{j}\|_1\leq\Gamma_j(r_j)$.

    Let $u_{\bs m}:=\mc P_{\bs m}u$ be an HOSVD approximation as in
    \eqref{eq:ubsr} with $m_j>r_j$ for all $j$. Then, since
    $\prod_{i\neq j}^d \mc P^{i}_{r_i}$ is
    orthogonal w.r.t. $\|\cdot\|_{e_j}$ and applying again
    Lemma \cref{lemma:h1errormd}, we have
    \begin{align*}
        &\|u_{\bs m}-u_{\bs r}\|_1^2\leq\sum_{j=1}^d\|(\mc P_{\bs m}-
        \mc P_{\bs r})u\|_{e_j}^2\\
        &\leq
        \sum_{j=1}^d\left[\left\|(\mc P_{m_j}^{j}-
        \mc P_{r_j}^{j})\prod_{i\neq j}^d\mc P^{i}_{r_i}u\right\|_{e_j}+
        \left\|\mc P_{m_j}^{j}\left(\prod_{i\neq j}^d\mc P^{i}_{m_i}-
        \prod_{i\neq j}^d\mc P^{i}_{r_i}\right)u\right\|_{e_j}
        \right]^2\\
        &\leq 2\sum_{j=1}^d
        \sum_{k=r_j+1}^{m_j}(\sigma_k^j)^2\left(1+
        \left\|\frac{d}{dx}\psi_k^j\right\|_0^2\right)+\Gamma^2_j(m_j)
        \sum_{i\neq j}\sum_{k=r_i+1}^{m_i}(\sigma_k^i)^2.
    \end{align*}
    If \eqref{eq:limitl2hosvd} holds,
    then $u_{\bs r}$ is a Cauchy sequence in $H^1$ and by uniqueness
    of the limit we must have $\|u-u_{\bs r}\|_1\rightarrow 0$.

    Finally, we show the bounds in \eqref{eq:l2hosvdbound}.
    The mapping $\bs \id=\bigotimes_{j=1}^d\id_j$ is an orthogonal projection
    in  the $\|\cdot\|_{e_j}$ norm and
    $\Im{\prod_{i\neq j}^d\mc P^{i}_{r_i}}\subset\Im{\bs\id}$.
    Thus, for any $1\leq j\leq d$
    \begin{align*}
        \left\|(\bs\id-\mc P_{\bs r})u\right\|^2_1&\geq
        \left\|(\bs \id-\mc P_{\bs r})u\right\|^2_{e_j}=
        \left\|\left[\bs\id-\left(\prod_{i\neq j}^d\mc P^{i}_{r_i}\right)
        \mc P^{j}_{r_j}\right]u\right\|^2_{e_j}\\
        &\geq
        \left\|(\bs\id-\mc P^{j}_{r_j})u\right\|^2_{e_j}
        =\sum_{k=r_j+1}^\infty(\sigma_k^j)^2
        \left(1+\left\|\frac{d}{dx}\psi_k^j\right\|_0^2\right),
    \end{align*}
    where the last equality is due to Lemma \cref{lemma:h1errormd}.
    This shows the lower bound in \eqref{eq:l2hosvdbound}.

    For the upper bound
    \begin{align*}
        \left\|(\bs \id-\mc P_{\bs r})u\right\|^2_1&\leq
        \sum_{j=1}^d\left\|(\bs\id-\mc P_{\bs r})u\right\|^2_{e_j}\\
        &=
        \sum_{j=1}^d\left\|(\bs\id-\prod_{i\neq j}^d\mc P^{i}_{r_i}
        )u\right\|^2_{e_j}+
        \left\|\prod_{i\neq j}^d\mc P^{i}_{r_i}
        (\bs\id-\mc P^{j}_{r_j})u\right\|^2_{e_j}\\
        &\leq
        \sum_{j=1}^d
        \left[
        \sum_{k=r_j+1}^\infty(\sigma_k^j)^2
        +
        \sum_{k=r_j+1}^\infty
        (\sigma^j_k)^2\left(1+\left\|\frac{d}{dx}\psi_k^j\right\|_0^2\right)
        \right]\\
        &=
        \sum_{j=1}^d\sum_{k=r_j+1}^\infty
        (\sigma_k^j)^2\left(2+\left\|\frac{d}{dx}\psi_k^j\right\|_0^2
        \right).
    \end{align*}
    This completes the proof.
\end{proof}

\bibliographystyle{acm}
\bibliography{literature}

\end{document}